\documentclass[11pt,a4paper,leqno]{article}
\usepackage{amsmath}
\usepackage{amsfonts}
\usepackage{amssymb}
\usepackage{amsthm}
\usepackage{graphicx}
\usepackage{hyperref}
\usepackage[margin=2.5cm]{geometry}

\newtheorem{theorem}{Theorem}[section]
\newtheorem{corollary}[theorem]{Corollary}
\newtheorem{lemma}[theorem]{Lemma}

\theoremstyle{remark}
\newtheorem{remark}[theorem]{Remark}

\title{Iterates of Blaschke products and Peano curves}

\author{Juan Jes\'us Donaire and 
 Artur Nicolau\thanks{The authors are supported in part by the Generalitat de Catalunya (grant 2017 SGR 395), the Spanish Ministerio de Ciencia e Innovaci\'on (project  MTM2017-85666-P) and the María de Maeztu Award.}  \\

{\small
\begin{tabular}{@{}c}
Departament de Matem\`atiques, 
Universitat Aut\`onoma de Barcelona\\
and\\ Centre de Recerca Matem\`atica,\\ 
08193 Barcelona\\
{\tt donaire@mat.uab.cat}\\
{\tt artur@mat.uab.cat}
\end{tabular}}}

\date{}

\begin{document}
\maketitle

\begin{abstract}
\noindent Let $f$ be a finite Blaschke product  with $f(0)=0$ which is not a rotation and let $f^{n}$ be its $n$-th iterate. Given a sequence  $\{a_{n}\}$ of complex numbers consider  $F= \sum a_n f^{n}$. If $\{a_n\}$ tends to $0$ but  $\sum |a_n| = \infty$, we prove that for any complex number $w$  there exists a point $\xi$ in the unit circle such that $\sum a_{n}f^{n}(\xi)$ converges and its sum is $w$. If $\sum |a_n| < \infty$ and the convergence is slow enough in a certain precise sense, then the image of the unit circle by $F$ has a non empty interior. The proofs are based on inductive constructions which use the beautiful interplay between the dynamics of $f$ as a selfmapping of the unit circle and those as a selfmapping of the unit disc. 

\end{abstract}

\section{Introduction}\label{sec1}
 
 A lacunary power series is a power series of the form
 \begin{equation}\label{eq1}
 F(z)=\sum_{n=1}^{\infty} a_{n} z^{k_{n}},
 \end{equation}
 where $\{a_{n}\}$ is a sequence of complex numbers and $\{k_{n}\}$ is a sequence of positive integers satisfying $\inf k_{n+1}/k_{n} >1$. The behavior of lacunary series has been extensively studied and it has been shown that in many senses, they behave as sums of independent random variables. If the coefficients~$\{a_{n}\}$ satisfy $\sum |a_{n}|=\infty$ but $a_{n}\to 0$ as $n\to\infty$, a theorem of Paley, proved by Weiss in~\cite{W}, says that for any $w\in\mathbb{C}$ there exists a point~$\xi$ in the unit circle~$\partial\mathbb{D}$ such that $\sum a_{n}\xi^{k_{n}} $ converges and its sum is~$w$. Salem and Zygmund proved that boundary values of certain lacunary series are Peano curves (\cite{SZ}). Their result was refined by Kahane, Weiss and Weiss who showed that if $\sum |a_{n}| <\infty$ but the convergence is slow enough (in a certain precise sense), then $F\colon \partial\mathbb{D}\to\mathbb{C}$ as defined in~\eqref{eq1} is a Peano curve, that is, $F(\partial\mathbb{D})$ contains a (non-degenerate) disc. See \cite{KWW}. More recent related results have
been proved by Bara\'nski (\cite{Ba}), Belov (\cite{Be}), Murai (\cite{Mu1}, \cite{Mu2}, \cite{Mu3}) and Younsi (\cite{Y}).

Let $f$ be a finite Blaschke product with $f(0)=0$ which is not a rotation and let $f^{n}$ denote its $n$-th iterate. The main purpose of this paper is to present analogous results for series of the form
\begin{equation}\label{eq2}
F(z)=\sum_{n=1}^{\infty} a_{n} f^{n}(z).
\end{equation}
It is worth mentioning that several recent results show that linear combinations of iterates, as defined in~\eqref{eq2}, behave as lacunary series. See \cite{NS} and \cite{N}. Our first result is a version of Paley's Theorem in this context.

\begin{theorem}\label{theo1.1}
Let $f$ be a finite Blaschke product with $f(0)=0$ which is not a rotation. Let $\{a_{n}\}$ be a sequence of complex numbers tending to~$0$ such that $\sum |a_{n}|=\infty$. Then for any $w\in\mathbb{C}$ there exists $\xi\in\partial\mathbb{D}$ such that $\sum a_{n}f^{n}(\xi)$ converges and $ \displaystyle{\sum_{n=1}^{\infty}} a_n f^{n}(\xi)=w$.
\end{theorem}

Let $m$ denote Lebesgue measure on the unit circle normalized so that $m(\partial\mathbb{D})=1$. It has been proved in~\cite{N} that if $\sum |a_{n}|^{2}<\infty$, then the series~\eqref{eq2} converges at almost every point of the unit circle. Conversely, if $\sum |a_{n}|^{2}=\infty$, then the series~\eqref{eq2} diverges at almost every point of the unit circle. Hence if the coefficients~$\{a_{n}\}$ tend to~$0$ but $\sum|a_{n}|^{2}=\infty$, Theorem~\ref{theo1.1} provides a set $E\subset \partial\mathbb{D}$ with $m(E)=0$ such that for any $w\in\mathbb{C}$ there exists $\xi\in E$ such that $\sum a_{n}f^{n}(\xi)$ converges and its sum is~$w$. Our next result says that under appropriate conditions on the coefficients~$\{a_{n}\}$, series of the form~\eqref{eq2} lead to Peano curves. 

\begin{theorem}\label{theo1.2}
Let $f$ be a finite Blaschke product with $f(0)=0$ which is not a rotation. Let $\{a_{n}\}$ be a sequence of complex numbers with $\sum |a_{n}|<\infty$. Assume
\begin{equation}\label{eq3}
\lim_{n\to\infty}\frac{|a_{n}|}{\sum\limits_{k>n}|a_{k}|}=0.
\end{equation}
Then $F= \displaystyle\sum_{n=1}^{\infty} a_n f^{n}\colon \partial\mathbb{D}\to\mathbb{C}$ is a Peano curve, that is, $F(\partial\mathbb{D})$ contains a (non-degenerate) disc.
\end{theorem}

The proofs of these results are based on delicate inductive constructions which use some techniques due to Weiss (\cite{W}) but we also need several new ideas which arise from the beautiful interplay between dynamical properties of a Blaschke product as a selfmapping of~$\partial\mathbb{D}$ and those as a selfmapping of the unit disc~$\mathbb{D}$. It is worth mentioning that no lacunarity assumption is needed in our results. The following more technical result plays a central role in our arguments.

\begin{theorem}\label{theo1.3}
Let $f$ be a finite Blaschke product with $f(0)=0$ which is not a rotation. Then there exist constants $\varepsilon=\varepsilon(f)>0$ and $ 0 < c=c(f) < 1$ such that the following statement holds. Let $M<N$ be positive integers, let $z\in\mathbb{D}$ with $|f^{M}(z)|<\varepsilon$ and let $\{a_{n}:M\le n\le N\}$ be a collection of complex numbers. Then there exists a point $\xi\in\partial\mathbb{D}$ with $|\xi-z|\le c^{-1} (1-|z|)$ such that
\begin{equation}\label{eq3.5}
\operatorname{Re}\left(\sum^{N}_{n=M}a_{n} f^{n}(\xi)\right)\ge c\sum^{N}_{n=M}|a_{n}|.
\end{equation}
\end{theorem}

The paper is organized as follows. Several auxiliary results are collected in Section~\ref{sec2}. These are used in Section~\ref{sec3} where we present the proofs of our main results. Finally, in Section~\ref{sec4} we prove a version of the classical Abel's Theorem in our context, sketch a proof of a generalization of Theorem~\ref{theo1.1} and conclude mentioning several open problems.

\section{Auxiliary results}\label{sec2}

Let $g$ be an analytic mapping from $\mathbb{D}$ into itself with $g(0)=0$  which is not a rotation. The classical Denjoy-Wolff Theorem says that the iterates $g^n$ converge uniformly to $0$ on compacts of $\mathbb{D}$. Actually Pommerenke proved the following exponential decay: there exist constants $0<a=a(f)<1$ such that $|g^n (z)| \leq a^n (1-|z|)^{-13}$ for any $z \in \mathbb{D}$. See \cite{P}. 

A finite Blaschke product $f$ is a finite product of automorphisms of~$\mathbb{D}$, that is,
$$
f(z)=\prod^{N}_{n=1}\frac{z-z_{n}}{1-\bar z_{n}z},\quad z\in\mathbb{C},
$$
where $z_{1},\dotsc,z_{n}\in\mathbb{D}$ are the zeros of~$f$. Observe that
\begin{equation}\label{eq4}
i\xi \frac{f'(\xi)}{f(\xi)} =\sum^{N}_{n=1}P(z_{n},\xi),\quad \xi\in\partial\mathbb{D},
\end{equation}
where $P(z,\xi)=(1-|z|^{2})|\xi-z|^{-2}$ is the Poisson kernel at the point~$z$. Hence
$$
f(e^{i\theta}) \overline{f(1)} =e^{i\psi (\theta)},\quad 0\le \theta\le 2\pi,
$$
where
$$
\psi(\theta)=\int^{\theta}_{0}\sum^{N}_{n=1} P(z_{n}, e^{it})\,dt.
$$
Note that $\psi$ is a real analytic branch of the argument of $f (e^{i\theta}) \overline{f(1)}$ which is increasing and satisfies
\begin{equation}\label{eq5}
f(e^{i\theta}) \overline{f(e^{i\varphi})} =\exp \left( i\int^{\theta}_{\varphi}\sum^{N}_{n=1}P(z_{n}, e^{it})\,dt\right),\quad 0\le \varphi\le \theta\le 2\pi.
\end{equation}
Let $f$ be a finite Blaschke product with $f(0)=0$ which is not a rotation. Note that \eqref{eq4} gives that $\min \{|f'(\xi)|:\xi\in\partial\mathbb{D}\}>1$ and the mapping $f\colon \partial\mathbb{D}\to \partial \mathbb{D}$ is expanding in the sense that $m(f(I))>m(I)$ for any arc $I\subset \partial\mathbb{D}$ with $m(I)<1$. Our first auxiliary result points in this direction.

\begin{lemma}\label{lem2.1}
Let $f$ be a finite Blaschke product with $f(0)=0$ which is not a rotation. Consider $K=K(f)=\min \{|f'(\xi)|:\xi\in\partial\mathbb{D}\}>1$. Let $N$ be a positive integer and let $I\subset \partial\mathbb{D}$ be an arc such that $m(f^{N}(I))=\delta<1$. Then
\begin{equation}\label{eq6}
|f^{k}(\xi)-f^{k}(\xi')|\le 2\pi \delta K^{k-N}
\end{equation}
for any $\xi,\xi'\in I$ and any integer $0 < k\le N$.
\end{lemma}

\begin{proof}
Since there exists an increasing continuous branch of $\operatorname{Arg} (f^{N})$ and $m(f^{N}(I))<1$, the mapping $f^{N} \colon I\to f^{N}(I)$ is one to one. Then for any integer $0<k\le N$ we have
\begin{equation*}
\begin{split}
\delta&= \int_{I}|(f^{N})'|\,dm= \int_{I} |(f^{N-k})' (f^{k})| |(f^{k})'|\,dm\\
&\ge K^{N-k} \int_{I}|(f^{k})'|\,dm =K^{N-k} m(f^{k}(I)).
\end{split}
\end{equation*}
We deduce that $m(f^{k}(I))\le \delta K^{k-N}$. Since $f^{k}(I)\subset \partial\mathbb{D}$ is an arc, the estimate~\eqref{eq6} follows.
\end{proof}

For future reference we now state two easy consequences of Lemma~\ref{lem2.1}.

\begin{corollary}\label{coro2.2}
Let $f$ be a finite Blaschke product with $f(0)=0$ which is not a rotation. Then there exists a constant~$c=c(f)>0$ such that if $I\subset \partial\mathbb{D}$ is an arc, $M<N$ are positive integers with $m(f^{N}(I))=\delta <1$ and $\{a_{n}: M\le n\le N\}$ is a collection of complex numbers, then
$$
\left| \sum^{N}_{n=M} a_{n}(f^{n}(\xi)-f^{n}(\xi'))\right| \le c\delta \left(\sum^{N}_{n=M}|a_{n}|^{2}\right)^{1/2},\quad \xi,\xi'\in I.
$$
\end{corollary}

\begin{proof}
Consider $K=\min \{|f'(\xi)|:\xi\in\partial\mathbb{D}\}$. Since $f(0)=0$ and $f$ is not a rotation, identity \eqref{eq4} gives $K>1$. Lemma~\ref{lem2.1} and Cauchy--Schwarz inequality give
$$
\sum^{N}_{n=M}|a_{n}| |f^{n}(\xi)-f^{n}(\xi')|\le \frac{2\pi \delta}{(K^{2} - 1)^{1/2}} \left(\sum^{N}_{n=M}|a_{n}|^{2}\right)^{1/2}.
$$
Taking $c=2\pi (K^{2} - 1)^{-1/2}$, the proof is completed.
\end{proof}

\begin{corollary}\label{coro2.3}
Let $f$ be a finite Blaschke product with $f(0)=0$ which is not a rotation. Fix $0<\delta<1$. Let $\{a_{n}\}$ be a sequence of complex numbers tending to~$0$. For $N=1,2,\dotsc$, let $I_{N}\subset\partial\mathbb{D}$ be an arc such that $m(f^{N}(I_{N}))=\delta$. Then
$$
\max \left\{\left| \sum^{N}_{n=1} a_{n} (f^{n} (\xi)-f^{n}(\xi'))\right| :\xi,\xi'\in I_{N}\right\}\underset{N\to\infty}{\longrightarrow} 0.
$$
\end{corollary}

\begin{proof}
Consider $K=\min \{|f'(\xi)|:\xi\in\partial\mathbb{D}\}$. As before, identity \eqref{eq4} gives $K>1$. Let $L=L(N)$ be a positive integer to be fixed later satisfying $1<L<N$. Since $m(f^{N}(I_{N}))=\delta$, Lemma~\ref{lem2.1} gives
\begin{equation*}
\begin{split}
\sum^{N}_{n=1}|a_{n}| |f^{n}(\xi)-f^{n}(\xi')| &\le 2\pi \delta \sup_{n}|a_{n}| \sum^{L}_{n=1}K^{n-N}+2\pi\delta  \sup_{n>L} |a_{n}|  \sum^{N}_{n=L+1}K^{n-N} \\*[5pt]
&\le 2\pi\delta  \sup_{n}|a_{n}| \frac{K^{L-N+1}}{K-1} +2\pi \delta \sup_{n>L}|a_{n}|  \frac{K}{K-1}
\end{split}
\end{equation*}
which tends to $0$ as $N$ tends to $\infty$, if $L$ is chosen such that both $L$ and $N-L$ tend to~$\infty$ as $N$ tends to~$\infty$.
\end{proof}

Given a point $z\in\mathbb{D}\backslash\{0\}$, let $I(z)$ denote the arc of the unit circle centered at~$z/|z|$ with $m(I(z))=1-|z|$. Conversely given an arc $I\subsetneq \partial\mathbb{D}$ let $z(I)$ be the point in~$\mathbb{D}$ satisfying $I(z(I))=I$. Let $\rho(z,w)$ denote the pseudohyperbolic distance between $z,w\in\mathbb{D}$ given by 
$$
\rho(z,w)=\frac{|z-w|}{|1-\bar wz|}.
$$

\begin{lemma}\label{lem2.4}
\begin{enumerate}
\item[(a)] For any $0<\gamma<1$, there exists $\delta=\delta(\gamma)>0$ such that if $f$ is a finite Blaschke product and $z\in\mathbb{D}$ satisfies $|f(z)|\le \gamma$, then $m(f(I(z)))\ge\delta$.

\item[(b)] Let $f$ be a finite Blaschke product with $f(0)=0$ which is not a rotation.~Then, given $0<\delta<1$, there exists $0<\gamma=\gamma(\delta,f)<1$ such that  if $N$ is a positive integer, $I\subset \partial\mathbb{D}$ is an arc with $m(f^{N}(I))=\delta$, then $|f^{N}(z(I))|\le\gamma$.
\end{enumerate}
\end{lemma}

\begin{remark}\label{rem2.5}
It is worth mentioning that the converse of the estimate in~(a) does not hold, that is, $|f(z)|$ could be arbitrarily close to~$1$ and $m(f(I(z)))\ge 1/2$. For instance, if $f(z)=(z-r)^{n}/(1-rz)^{n}$, where $0<r<1$ and $n>2$, we have $|f(s)|=|(s-r)/(1-rs)|^{n}$, which could be arbitrarily close to~$1$ if $\rho (s,r)$ is sufficiently close to~$1$, while $f(I(s))=\partial\mathbb{D}$ if $s<r$. Part~(b) says that the converse of the estimate~(a) holds uniformly for iterates of a finite Blaschke product~$f$, if the constant $\gamma$ is allowed to depend on~$f$.
\end{remark}

\begin{proof}
(a) Let $\{z_{n}\}$ be the zeros of~$f$. We can assume that the mapping $f\colon I(z)\to f(I(z))$ is one to one since otherwise using \eqref{eq5} we would deduce that $f(I(z))=\partial \mathbb{D}$. We can also assume $\inf\{\rho (z,z_{n}):n\ge 1\}\ge 1/2$. The identity~\eqref{eq4} gives 
$$
m(f(I(z)))=\int_{I(z)} \sum_n P(z_{n},\xi)\,dm(\xi).
$$
Note that there exists a universal constant $c_1>0$ such that $|\xi-z_{n}|\le c_1 |1-z_{n}\bar z|$ for any $\xi\in I(z)$ and any~$n$. We deduce
\begin{equation}\label{eq6,1}
    m(f(I(z)))\ge c_1^{-2} \sum_{n}\frac{1-|z_{n}|^{2}}{|1-z_{n}\bar z|^{2}} (1-|z|).
\end{equation}
Since $\inf \{\rho (z,z_{n}):n\ge 1\}\ge 1/2$, the elementary estimate $-\log x\le c_{2}(1-x^{2})$, $1/2\le x\le 1$, provides a universal constant $c_{3}>0$ such that
\begin{equation}\label{eq6,2}
  \sum_{n}\frac{(1-|z_{n}|^{2})(1-|z|)}{|1-z_{n}\bar z|^{2}}\ge c_{3}\log |f(z)|^{-1}.  
\end{equation}
This finishes the proof of (a). We now prove (b). We first show that there exists a constant~$c_{4}=c_{4}(f, \delta)>1$ such that
\begin{equation}\label{eq7}
|(f^{N})'(\xi)|\le c_{4}|(f^{N})'(\xi^{*})|,\quad \xi,\xi^{*}\in I,\quad N\ge 1.
\end{equation}
Using that $|\mathrm{log} x-\log y| \le |x-y|$ for any $x,y>1$, we obtain
\begin{equation*}
\begin{split}
\left|\log \frac{|(f^{N})'(\xi)|}{|(f^{N})'(\xi^{*})|} \right|&=\left|\sum^{N-1}_{k=1}\log 
\frac{|f'(f^{k}(\xi))|}{|f'(f^{k}(\xi^{*}))|}\right|\\*[5pt]
&\le \sum^{N-1}_{k=1}|f'(f^{k}(\xi))-f'(f^{k}(\xi^{*}))|\le c_{5}\sum^{N-1}_{k=1}|f^{k}(\xi)-f^{k}(\xi^{*})|,
\end{split}
\end{equation*}
where $c_{5}=\max\{|f''(\xi)|: \xi\in \partial\mathbb{D}\}$. Since $m(f^{N}(I))=\delta <1$, Lemma~\ref{lem2.1} gives the estimate~\eqref{eq7}.

Let $\xi(I)$ be the center of~$I$. Applying \eqref{eq7}, for any measurable subset $J\subset I$, we have
$$
m(f^{N}(J))=\int_{J}|(f^{N})'|\,dm \ge c_{4}^{-1} |(f^{N})'(\xi(I))|m(J)\ge c_{4}^{-2} m(f^{N}(I))\frac{m(J)}{m(I)}=c_{4}^{-2} \delta \frac{m(J)}{m(I)}.
$$
We deduce that there exists a constant $c_{6}=c_{6}(\delta,f)>0$ such that
\begin{equation}\label{eq8}
\frac{1}{m(I)} \int_{I}|f^{N}-f^{N}(z(I))|^{2}\,dm\ge c_{6}.
\end{equation}
Since there exists a universal constant $c_{7}>0$ such that $P(z(I),\xi)\ge c_{7}m(I)^{-1}$ for any $\xi\in I$, we obtain 
\begin{equation*}
\begin{split}
\frac{1}{m(I)}\int_{I}|f^{N}-f^{N}(z(I))|^{2}\,dm&\le c_{7}^{-1} \int_{\partial\mathbb{D}} |f^{N}(\xi)-f^{N}(z(I))|^{2}P(z(I),\xi)\,dm(\xi)\\
&=c_{7}^{-1} (1-|f^{N}(z(I))|^{2}).
\end{split}
\end{equation*}
Using \eqref{eq8} we deduce $1-|f^{N}(z(I))|^{2}\ge c_{6}c_{7}$ and the proof is completed.
\end{proof}

Next auxiliary result will be used in Section 4 and it is not needed in the proofs of our main results. 

\begin{lemma}\label{lem4.0}
Let $f$ be a finite Blaschke product with $f(0)=0$ which is not a rotation. Fix $0< \delta < ( 10 \max \{|f'(\xi)| : \xi \in \partial \mathbb{D}\})^{-1}$. Then there exits a constant $C=C(\delta , f) >0$ such that the following statement holds. Given $z \in \mathbb{D}$ let $N(z)$ be the smallest positive integer such that $m(f^{N(z)}(I(z)))\ge \delta$. Then 
\begin{equation}\label{eq38}
\sum^{N(z)}_{k=1}|f^{k}(\xi)-f^{k}(z)|\le C, \quad \xi \in I(z). 
\end{equation}
\end{lemma}
\begin{proof}
Let $K_1=\max\{|f'(\xi)|\, :\, \xi\in\partial\mathbb{D}\}$. Note that $N(z)\to\infty$ as $|z|\to 1$ and $m(f^{N(z)}(I(z)))\le \delta K_{1}$.~Part~(b) of Lemma~\ref{lem2.4} gives a constant $0<\gamma=\gamma(\delta,f)<1$ such that $|f^{N(z)}(z)|\le \gamma$.~Note that $f^{N(z)}$ can not have zeros at pseudohyperbolic distance less than~$1/2$ of~$z$, because in this case $m(f^{N(z)}(I(z)))$ would be large. Then, estimates \eqref{eq6,1} and \eqref{eq6,2} in the proof of part~(a) of Lemma~\ref{lem2.4} give that $\delta K_{1}\ge m(f^{N(z)}(I(z)))\ge c_{1}^{-2}c_{3}\log |f^{N(z)}(z)|^{-1}$. Hence
\begin{equation}\label{eq36}
e^{-\delta K_{1}c_{1}^{2}/c_{3}}\le |f^{N(z)}(z)|\le\gamma.
\end{equation}
Pommerenke estimates of the Denjoy--Wolff Theorem (Lemma~2 of \cite{P}) provide constants $C(\gamma) >0$ and  $0<a=a(f)<1$ such that
\begin{equation}\label{eq37}
|f^{n}(z)|\le C(\gamma) a^{n-N(z)},\quad n\ge N(z).
\end{equation}
Let $K=\min\{|f'(\xi)|:\xi\in\partial\mathbb{D}\}$.~Since $f(0)=0$ and $f$ is not a rotation, identity \eqref{eq4} gives $K>1$. Lemma~\ref{lem2.1} gives 
$$
|f^{k}(\xi)-f^{k}(\xi')|\le 2\pi\delta K_{1}K^{k-N(z)},\quad \xi,\xi'\in I(z),\quad 0 <  k\le N(z).
$$
Hence $|f^{k}(\xi)-f^{k}(z)|\le |f^{k}(\xi')-f^{k}(z)|+2\pi\delta K_{1}K^{k-N(z)}$, for any $\xi, \xi' \in I(z)$ and $0< k \leq N(z)$. Integrating over $\xi'\in I(z)$ and using that there exists a constant $c_4 >0$ such that  $P(z, \xi) \geq c_4 / (1-|z|), \xi \in I(z)$, we deduce
\begin{equation*}
\begin{split}
|f^{k}(\xi)-f^{k}(z)|&\le (1-|z|)^{-1 }\int_{I(z)} |f^{k}(\xi')-f^{k}(z)| \,dm(\xi')+2\pi\delta K_{1}K^{k-N(z)}\\
&\le c_4^{-1} \left(\int_{\partial\mathbb{D}}|f^{k}(\xi')-f^{k}(z)|^{2} P (z , \xi')\,dm (\xi')\right)^{1/2}+2\pi\delta K_1 K^{k-N(z)}\\
&=c_4^{-1} (1-|f^{k}(z)|^{2})^{1/2}+2\pi\delta K_{1}K^{k-N(z)},\; \xi \in I(z), \; 0 <  k\leq N(z).
\end{split}
\end{equation*}
Using Lemma~2.1(a) of \cite{N} and \eqref{eq36} one finds  constants $c_5 = c_5 (\delta) >0$ and  $0<b = b(f)<1$ such that $1-|f^{k}(z)|^{2}\le c_5 b^{N(z)-k}$, if $0 <  k\leq N(z)$. Hence there exists a constant $C=C(\delta, f)>0$ such that
\begin{equation}\label{eq38}
\sum^{N(z)}_{k=1}|f^{k}(\xi)-f^{k}(z)|\le C, \quad \xi \in I(z). 
\end{equation}
\end{proof}

Next auxiliary result is Lemma~3.3 of \cite{N}.

\begin{lemma}\label{lem2.6}
Let $f$ be an inner function with $f(0)=0$ which is not a rotation.~Then there exist two constants $0<\varepsilon=\varepsilon(f)<1$ and $0<c=c(f)<1$ such that the following statement holds. Let $M<N$ be positive integers, $z\in\mathbb{D}$ satisfying $|f^{M}(z)|<\varepsilon$ and $\{a_{n}:M\le n\le N\}$ a collection of complex numbers. Then there exists
a set $E=E(z,\{a_{n}\})\subset \partial\mathbb{D}$ with $E\subset c^{-1}I(z)$ such that
$$
\operatorname{Re} \left(\sum^{N}_{n=M}a_{n}f^{n}(\xi)\right)\ge c\left(\sum^{N}_{n=M}|a_{n}|^{2}\right)^{1/2},\quad \xi\in E.
$$
\end{lemma}

Our last auxiliary result is elementary and it is stated for future reference.

\begin{lemma}\label{lem2.7}
Let $\{b_{n}\}$ be a sequence of positive numbers with $\sum b_{n}<\infty$. Consider $S_n= \sum\limits_{k>n}b_{k}$. Assume
$$
\lim_{n\to\infty}\frac{b_{n}}{S_n}=0.
$$
Fix $K>1$. Then 
$$
\lim_{N \to \infty} \frac{\sum_{n=1}^{N}b_{n}K^{n-N}}{S_N} =0
$$
\end{lemma}

\begin{proof}
Fix $\varepsilon >0$ such that $K(1+ \varepsilon)^{-1} >1$. By assumption there exists $n_0 = n_0 (\varepsilon) >0$ such that $S_n \geq (1+ \varepsilon)^{-1} S_{n-1} $ for any $n \geq n_0$. Hence $S_n \geq (1+ \varepsilon)^{-n + n_0} S_{n_0}$, $n \geq n_0$. We deduce that 
\begin{equation}\label{101}
    \lim_{n \to \infty} K^n S_n = \infty .
\end{equation}
Since $b_n \leq \varepsilon S_n$ for $n \geq  n_0$ and $S_n \leq (1 + \varepsilon)^{N - n} S_N$ for $n_0 \leq n \leq N$, we have 
$$
\sum^{N}_{n=n_0}b_{n}K^{n-N} \leq \varepsilon S_N \sum^{N}_{n=n_0} \left(\frac{K}{1+ \varepsilon} \right)^{n-N}.
$$
Now the identity \eqref{101} finishes the proof.  
\end{proof}

\section{Proofs of the main results}\label{sec3}

This section is devoted to the proofs of Theorems~\ref{theo1.1}, \ref{theo1.2} and \ref{theo1.3}. It is worth mentioning that under additional assumptions on the function~$f$, one can give a short proof of Theorem~\ref{theo1.3}. Actually, let $f$ be a finite Blaschke product and assume that there exists a constant $\varepsilon>0$ such that
\begin{equation}\label{eq10}
f(I)=\partial\mathbb{D}
\end{equation}
for any arc $I\subset \partial\mathbb{D}$ with $m(I)\ge 1/2-\varepsilon$. For instance, according to \eqref{eq5}, any function $f$ of the form $f(z)=z^{2}g(z)$, where $g$ is a non-constant Blaschke product, satisfies \eqref{eq10}. Let $\{a_{n}\}$  be a sequence of complex numbers.~For any $n$  such that $a_{n}\ne 0$, consider the closed arc~$I_{n}$ centered at $\overline{a}_{n}/|a_{n}|$ with $m(I_{n})=1/2-\varepsilon$. Note that there exists a constant $c=c(\varepsilon)>0$ such that
\begin{equation}\label{eq11}
\operatorname{Re}\left(\sum_{n}a_{n}\xi_{n}\right)\ge c\sum_{n}|a_{n}|,
\end{equation}
for any choice of points $\xi_{n}\in I_{n}$, $n\ge 1$. Take $J_1 = I_1$.~We now construct inductively a sequence of nested closed arcs~$J_{n}$ such that $f^{n-1}(J_{n})=I_{n}$, $n>1$. Since $f(I_{1})=\partial\mathbb{D}$, there exists an arc~$J_{2}\subset I_{1}$ with $f(J_{2})=I_{2}$. Assume $J_{2},\dotsc,J_{n-1}$ have been constructed satisfying $f^{n-2} (J_{n-1})=I_{n-1}$. Then \eqref{eq10} gives $f^{n-1}(J_{n-1})=f(I_{n-1})=\partial\mathbb{D}$ and we can find an arc~$J_{n}\subset J_{n-1}$ such that $f^{n-1}(J_{n})=I_{n}$. This finishes the construction of the nested sequence of arcs~$\{J_{n}\}$. Pick $\xi\in \partial\mathbb{D}$ such that $f(\xi)\in \bigcap\limits_{n}J_{n}$. Since $f^{n}(\xi)\in I_{n}$ for any $n\ge 1$, estimate~\eqref{eq11} gives 
$$
\operatorname{Re}\left(\sum_{n}a_{n}f^{n}(\xi)\right)\ge c\sum_{n}|a_{n}|.
$$
This finishes the proof of Theorem~\ref{theo1.3} under the additional assumption~\eqref{eq10}.~However there are finite Blaschke products~$f$ which do not satisfy \eqref{eq10} and the proof of Theorem~\ref{theo1.3} requires different ideas.

\begin{proof}[Proof of Theorem~\ref{theo1.3}]

The proof of Theorem~\ref{theo1.3} is organized in three steps.

\vspace*{7pt}
\noindent
\emph{1.\ Splitting.} We use an idea of M.~Weiss (\cite{W}). Let $T'<T$ be two (large) positive integers to be fixed later such that $T/T'$ is also an integer. Since one can obviously add terms with vanishing coefficients to the left-hand side term of (\ref{eq3.5}), one can assume that $N-M+1$ is a multiple of $T$. Split the sum $F=\sum\limits_{M}^{N}a_{n}f^{n}$ into blocks of length~$T$, that is, $F=\sum\limits_{k\ge 0}G_{k}$, where
$$
G_{k}=\sum_{n=0}^{T-1} a_{M+kT+n} f^{M+kT+n}, \quad k=0,1,2,\dotsc
$$
Next split $G_{2k}$, $k\ge 1$, into successive blocks of length~$T'$ and pick the subblock such that the sum of the modulus of the coefficients is the least, that is, pick $\mathcal{S}_{k}$ a subset of $T'$ consecutive integers in~$[M+2kT, M+(2k+1)T)$ of the form $\{M+ 2kT + jT' , \ldots , M+2kT + (j+1)T' -1 \}$ such that
$$
\sum_{n\in \mathcal{S}_{k}}|a_{n}| \le \sum_{\ell\in\mathcal{S}}|a_{\ell}|,
$$
for any other subset~$\mathcal{S}$ of $T'$ consecutive integers in $[M+2kT,M+(2k+1)T)$ of the same form. Since the number of disjoint subblocks of this type is $T/T'$ we deduce
\begin{equation}\label{eq12}
\sum_{n\in \mathcal{S}_{k}}|a_{n}|\le \frac{T'}{T}\sum^{T-1}_{n=0} |a_{M+2kT+n}|,\quad k\ge 1.
\end{equation}
The corresponding block
$$
{S}_{k}=\sum_{n\in \mathcal{S}_{k}} a_{n}f^{n},\quad k\ge 1
$$
will be called a short block. The long blocks are defined as the blocks between two consecutive short blocks as well as the first and last block. More concretely if $\mathcal{S}_{k}=\{n\in\mathbb{Z}: N_{k}<n<M_{k+1}\}$, $k\ge 1$, then
$$
L_{1}=\sum^{N_{1}}_{n=M} a_{n} f^{n},\quad L_{k}=\sum^{N_{k}}_{n=M_{k}}a_{n}f^{n},\quad k>1,
$$
and
$$
S_{k}=\sum^{M_{k+1}-1}_{n=N_{k}+1} a_{n}f^{n},\quad k\ge 1.
$$
Write $M_1 = M$ and $\mathcal{L}_{k}=\{n\in\mathbb{Z}: M_{k}\le n\le N_{k}\}$, $k \geq 1$. Note that for any $k\ge 1$, the set $\mathcal{S}_{k}$ has $T'$~indexes while the number of indices in $\mathcal{L}_{k}$ is between $T$ and $3T$. Note also that \eqref{eq12} gives
\begin{equation}\label{eq13}
\sum_{k}\sum_{n\in \mathcal{L}_{k}} |a_{n}| \ge \left(1-\frac{T'}{T}\right)\sum^{N}_{n=M}|a_{n}|.
\end{equation}

\vspace*{7pt}
\noindent
\emph{2.\ The inductive construction.}
Let $0 < \varepsilon =  \varepsilon (f)< 1$ and $0 < c =  c (f)< 1$ be the constants appearing in Lemma  \ref{lem2.6}.
We will show that there exist constants $c_{0}=c_{0}(f)>0$ and $0<\gamma=\gamma(f)<1$ and a sequence of nested closed arcs~$I_{k}\subset c^{-1} I(z)$, such that
\begin{align}
|f^{N_{k}}(z(I_{k}))| &\le \gamma,\label{eq14}\\
\operatorname{Re} \sum^{k}_{j=1} L_{j}(\xi) &\ge \frac{c_{0}}{T^{1/2}}\sum^{k}_{j=1} \sum_{n\in\mathcal{L}_{j}}|a_{n}|,\quad \xi\in I_{k}.\label{eq15}
\end{align}
We argue by induction. Since $|f^{M}(z)|<\varepsilon$, Lemma~\ref{lem2.6} provides a point $\xi_{1}\in c^{-1}I(z)$ with
\begin{equation}\label{eq16}
\operatorname{Re} L_{1}(\xi_{1})\ge c\left(\sum^{N_{1}}_{n=M} |a_{n}|^{2}\right)^{1/2}.
\end{equation}
Note that $|f^{N_{1}}(z)|\le |f^{M}(z)|< \varepsilon$. By part~(a) of Lemma~\ref{lem2.4}, there exists a constant $\delta_{0}=\delta_{0}(\varepsilon)>0$ such that $m(f^{N_{1}}(I(z))) \ge \delta_{0}$. Fix $0<\delta_{1}<\delta_{0}/2$ and pick an arc~$I_{1}$ with $\xi_{1}\in I_{1}\subset c^{-1}I(z)$ with $m(f^{N_{1}}(I_{1}))=\delta_{1}$. If $\delta_{1}>0$ is chosen sufficiently small, Corollary~\ref{coro2.2} and \eqref{eq16} give
$$
\operatorname{Re} L_{1}(\xi)\ge \frac{c}{2} \left(\sum_{n=M}^{N_1}|a_{n}|^{2}\right)^{1/2},\quad \xi\in I_{1}.
$$
Since $\mathcal{L}_{1}$ has at most $3T$~indexes, Cauchy--Schwarz's inequality gives 
$$
\sum_{n\in\mathcal{L}_{1}}|a_{n}| \le (3T)^{1/2}\left(\sum_{n\in\mathcal{L}_{1}}|a_{n}|^{2}\right)^{1/2}
$$
and \eqref{eq15} holds for $k=1$ if we pick $0<c_{0}<c/2\sqrt{3}$.~Since $m(f^{N_{1}}(I_{1}))=\delta_{1}$, part~(b) of Lemma~\ref{lem2.4} provides a constant $\gamma=\gamma (\delta_{1},f)<1$ such that the estimate~\eqref{eq14} holds for $k=1$. Assume now that the arcs $I_{k}\subset I_{k-1}\subset\dotsb \subset I_{1}$ have been constructed so that \eqref{eq14} and \eqref{eq15} hold. Next we will construct $I_{k+1}$. Fix a constant $0<c_{1}=c_{1}(c)<1$ such that there exists a point $z_{k}^{*}\in\mathbb{D}$ with $\rho(z_{k}^{*}, z(I_{k}))\le c_{1}$ with $c^{-1}I(z_{k}^{*})\subset I_{k}$. Since $|f^{N_{k}}(z(I_{k}))|\le\gamma$, Schwarz's Lemma gives a constant $0<\gamma_{1}=\gamma_{1}(\gamma,c_{1})<1$ with $|f^{N_{k}}(z_{k}^{*})|\le\gamma_{1}$. Recall that $M_{k+1}=N_{k}+T'$. Since $f^{n}\to 0$ uniformly on compacts of~$\mathbb{D}$, we can choose a positive integer $T' = T' (f)$ sufficiently large but independent of $k$, so that
$$
|f^{M_{k+1}}(z_{k}^{*})|\le \varepsilon.
$$
Apply Lemma~\ref{lem2.6} to find a point $\xi_{k}^{*}\in c^{-1}I(z_{k}^{*})\subset I_{k}$ such that
\begin{equation}\label{eq17}
\operatorname{Re} (L_{k+1}(\xi_{k}^{*}))\ge c\left(\sum_{n\in\mathcal{L}_{k+1}}|a_{n}|^{2}\right)^{1/2}.
\end{equation}
Note that since $f(0)=0$, Schwarz's lemma gives $|f^{N_{k+1}}(z_{k}^{*})|\le |f^{M_{k+1}}(z_{k}^{*})|\le\varepsilon$. By part~(a) of Lemma~\ref{lem2.4}, there exists a constant~$\delta_{0}=\delta_{0}(\varepsilon)>0$ such that $m(f^{N_{k+1}}(I(z_{k}^{*})))\ge \delta_{0}$.~Fix $0<\delta_{1}<\delta_{0}/2$ and pick an arc~$I_{k+1}$ with $\xi_{k}^{*}\in I_{k+1}\subset I_{k}$ such that $m(f^{N_{k+1}}(I_{k+1}))=\delta_{1}$. If $\delta_{1}>0$ is chosen sufficiently small, Corollary~\ref{coro2.2} and estimate~\eqref{eq17} give
$$
\operatorname{Re} (L_{k+1}(\xi))\ge \frac{c}{2} \left(\sum_{n\in \mathcal{L}_{k+1}}|a_{n}|^{2}\right)^{1/2},\quad \xi\in I_{k+1}.
$$
Since $\mathcal{L}_{k+1}$ has at most $3T$~indexes, Cauchy--Schwarz's inequality gives
$$
\sum_{n\in\mathcal{L}_{k+1}}|a_{n}|\le (3T)^{1/2}\left( \sum_{n\in\mathcal{L}_{k+1}}|a_{n}|^{2}\right)^{1/2}
$$
and we deduce
$$
\operatorname{Re}L_{k+1}(\xi)\ge \frac{c}{2}\frac{1}{(3T)^{1/2}}\sum_{n\in \mathcal{L}_{k+1}}|a_{n}|,\quad \xi\in I_{k+1}.
$$
If we pick $0<c_{0}<c/2\sqrt{3}$, the inductive assumption gives
$$
\operatorname{Re} \sum^{k+1}_{j=1}L_{j} (\xi)\ge \frac{c_{0}}{T^{1/2}}\sum^{k+1}_{j=1} \sum_{n\in\mathcal{L}_{j}}|a_{n}|,\quad \xi\in I_{k+1}
$$
which is \eqref{eq15} for the index~$k+1$. Since $m(f^{N_{k+1}} (I_{k+1}))=\delta_{1}$, the estimate~\eqref{eq14} for the index~$k+1$, follows from part~(b) of Lemma~\ref{lem2.4}. This finishes the proof of the existence of the nested sequence of arcs~$I_{k}$ satisfying \eqref{eq14} and \eqref{eq15}.

\vspace*{7pt}
\noindent
\emph{3.\ The final argument.}
Pick $\xi\in \bigcap\limits_{k} I_{k}$. The estimate~\eqref{eq15} gives
$$
\operatorname{Re} \left(\sum^{k}_{j=1}L_{j}(\xi) +\sum^{k-1}_{j=1}S_{j}(\xi)\right)\ge \frac{c_{0}}{T^{1/2}} \sum^{k}_{j=1}\sum_{n\in\mathcal L_{j}}|a_{n}|-\sum^{k-1}_{j=1}\sum_{n\in\mathcal{S}_{j}}|a_{n}|,\quad k > 1.
$$
Since $\sum_{n=M}^N a_{n}f^{n}$ is a sum of long and short blocks, the estimates~\eqref{eq12} and \eqref{eq13} give
$$
\operatorname{Re}\left(\sum_{n=M}^N a_{n}f^{n}(\xi)\right) \ge \frac{c_{0}}{T^{1/2}}\left(1-\frac{T'}{T}\right) \sum_{n=M}^N|a_{n}|-\frac{T'}{T}\sum_{n=M}^N|a_{n}|.
$$
Choosing $T$ such that $T'/T^{1/2}< c_{0}/4$, we deduce
$$
\operatorname{Re}\left(\sum_{n=M}^N a_{n}f^{n}(\xi)\right) \ge \frac{c_{0}}{2T^{1/2}}\sum_{n=M}^N|a_{n}|.
$$
This finishes the proof.
\end{proof}

The proof of Theorem~\ref{theo1.1} uses the following easy consequence of Theorem~\ref{theo1.3}. 

\begin{corollary}\label{coro3.1}
Let $f$ be a finite Blaschke product with $f(0)=0$ which is not a rotation. Then there exists a constant $0<\eta=\eta(f)<1$ such that the following statement holds.~Let $M<N$ be positive integers, let $z\in\mathbb{D}$ with $|f^{M}(z)|\le \eta$, let $\{a_{n}:M\le n\le N\}$ be a collection of complex numbers and let $w\in\mathbb{C}$ with
\begin{equation}\label{eq17.5}
\sum^{N}_{n=M} |a_{n}| \le \eta |w|.
\end{equation}
Then there exists a point $\xi\in \eta^{-1}I(z)$ such that
$$
\left| w-\sum^{N}_{n=M}a_{n}f^{n}(\xi)\right|\le |w|-\eta \sum^{N}_{n=M}|a_{n}|.
$$
\end{corollary}

\begin{proof}
We can assume $w=1$. Let $0< \varepsilon = \varepsilon (f)<1$ and $0< c = c(f) <1$ be the constants appearing in Theorem~\ref{theo1.3}. Pick $0< \eta < \min \{\varepsilon , c\}$. Applying Theorem~\ref{theo1.3} one finds a point $\xi\in c^{-1}I(z)$ such that
\begin{equation}\label{eq18}
\operatorname{Re} \left(\sum^{N}_{n=M} a_{n}f^{n}(\xi)\right)\ge c \sum^{N}_{n=M}|a_{n}|.
\end{equation}
Write $A=\sum\limits_{n=M}^{N}|a_{n}|$. Since $A\le \eta < 1$, we have
$$
\left| 1-\sum^{N}_{n=M}a_{n} f^{n}(\xi)\right|^{2}\le (1-cA)^{2}+A^{2}-(cA)^{2}= 1-A (2c-A).
$$
See Figure 1. Using the elementary estimate $\sqrt{1-x}\le 1-x/2$, $0\le x<1$, we deduce that  $\bigl|1-\sum\limits_{n=M}^{N}a_{n}f^{n}(\xi)\bigr|\le 1-A(c-A/2)\le 1-\eta A$ if $0<\eta <2c/3$.

\begin{figure}[h]
\begin{center}
\includegraphics{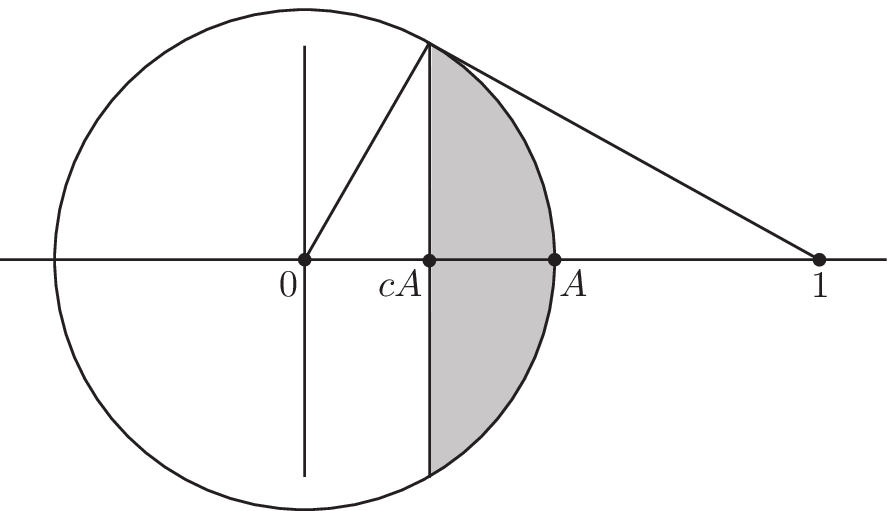}
\end{center}
\end{figure}

\hspace{3cm} Figure 1: $\sum a_{n} f^{n}(\xi)$ lies in the shadowed region.
\end{proof}

We are now ready to prove Theorem~\ref{theo1.1}. We will again combine some ideas of Weiss (\cite{W}) with the interplay between dynamical properties of $f$ as a selfmapping of $\mathbb{D}$ with those as a selfmapping of $\partial \mathbb{D}$.

\begin{proof}[Proof of Theorem~\ref{theo1.1}]
Fix $w\in\mathbb{C}$. We first divide the sum $\sum a_{n}f^{n}$ into blocks
$$
P_{j}=\sum_{n\in\mathcal{P}_{j}} a_{n}f^{n},\quad P_{j}^{*}=\sum_{n\in\mathcal{P}_{j}^{*}}a_{n}f^{n},
$$
where $\mathcal{P}_{j}=\{n\in\mathbb{Z}: M_{j}\le n\le N_{j}\}$, $M_{1}=1$ and $\mathcal{P}_{j}^{*}=\{n\in\mathbb{Z}:N_{j}<n<M_{j+1}\}$, $j\ge 1$. This splitting is similar to the decomposition used in the proof of Theorem~\ref{theo1.3} but this time the number of indices in $\mathcal{P}_{j}$ may grow as $j\to\infty$ while, as before, the number of indices in $\mathcal{P}_{j}^{*}$ is a large number independent of~$j$ to be chosen later. In other words, the blocks will be chosen in such a way that $M_{j+1}-N_{j}=T= T(f)$ is large but independent of~$j$, while $N_{j}-M_{j}$ may tend to infinity as $j\to\infty$. Write
$$
F_{k}=\sum^{k}_{j=1}P_{j}+\sum^{k-1}_{j=1}P_{j}^{*}=\sum^{N_{k}}_{n=1}a_{n}f^{n}.
$$
The blocks $P_{j}$ and $P_{j}^{*}$ will be constructed inductively, as will a sequence of positive numbers~$\delta_{k}$ tending to $0$, and a nested sequence of closed arcs~$I_{k}\subset \partial\mathbb{D}$ in such a way that there exists a constant $0< \gamma = \gamma (f) <1$ such that if $d_{k}=\min \{|F_{k}(\xi)-w|: \xi\in I_{k}\}$, the following estimates hold
\begin{align}
&d_{k+1}\le \gamma d_{k}+\delta_{k},\label{eq19}\\*[5pt]
&\sum_{n\in\mathcal{P}_{k}^{*}\cup \mathcal{P}_{k+1}}|a_{n}|\underset{k\to\infty}{\longrightarrow} 0,\label{eq20}\\*[5pt]
&m(f^{N_{k}}(I_{k}))= 1/2.\label{eq21}
\end{align}

Once this inductive construction is done, the result follows easily. Actually pick $\xi\in \bigcap\limits_{j}I_{j}$. Note that \eqref{eq19} gives $d_{k}\to 0$ as $k\to\infty$.~Then, identity~\eqref{eq21} and Corollary~\ref{coro2.3} give that $F_{k}(\xi)\to w$ as $k\to\infty$. Applying \eqref{eq20} we deduce that $\sum\limits_{n}a_{n}f^{n}(\xi)$ converges and its sum is~$w$. 

We now explain the construction of the blocks~$P_{j}$, $P_{j}^{*}$ and the arcs~$I_{j}$. For $k=1$ one can pick any partial sum~$P_{1}$ and any arc~$I_{1}\subset\partial\mathbb{D}$ with $m(f^{N_{1}}(I_{1}))= 1/2$. Assume by induction that the blocks $P_{1},P_{1}^{*},\dotsc,P_{k-1},P_{k-1}^{*},P_{k}$ and the arcs~$I_{k}\subset I_{k-1}\subset \dotsb\subset I_{1}$ have been chosen.~We will now define $P_{k}^{*}$, $P_{k+1}$ and the arc~$I_{k+1} \subset I_k$ and show that the estimates~\eqref{eq19}, \eqref{eq20} and \eqref{eq21} hold.~Let $\eta = \eta (f) >0$ be the constant appearing in Corollary~\ref{coro3.1}.~Fix a constant $0<c_{1}=c_{1}(\eta)<1$ such that there exists $z_{k}^{*} \in \mathbb{D}$ with $\rho(z_{k}^{*}, z(I_{k}))\le c_{1}$ such that $\eta^{-1} I(z_{k}^{*})\subset I_{k}$. Observe that \eqref{eq21} and part~(b) of Lemma~\ref{lem2.4} provide a constant~$0<\gamma_{0}=\gamma_{0}(f)<1$ such that $|f^{N_{k}}(z(I_{k}))|\le \gamma_{0}$. By Schwarz's Lemma, there exists a constant $0<\gamma_{1}=\gamma_{1}(\gamma_{0},\eta)<1$ such that $|f^{N_{k}}(z_{k}^{*})|\le \gamma_{1}$. Since $f^{n}\to 0$ uniformly on compacts of~$\mathbb{D}$, if the number of indices~$T= T(f)$ of~$P_{k}^{*}$ is chosen sufficiently large (but independent of~$k$), we deduce $|f^{M_{k+1}}(z_{k}^{*})|\le \eta$, where $M_{k+1}=N_k+T$. This defines $P_{k}^{*}$. This choice of $M_{k+1}$ allows to apply Corollary~\ref{coro3.1} to the point~$z_{k}^{*}$ and any partial sum starting at the index~$M_{k+1}$. To define the next block~$P_{k+1}$ we distinguish two cases.

\vspace*{7pt}
\noindent
(I) Assume $\max\limits_{j\ge N_{k}}|a_{j}|\le \eta d_{k}/4$. In this case, let $N_{k+1}$ be the minimal integer bigger than~$M_{k+1}$ such that
$$
\sum^{N_{k+1}}_{n=M_{k+1}}|a_{n}|\ge \eta d_{k}/2.
$$
This defines the block $P_{k+1}$. Since $|a_{N_{k+1}}| \leq \eta d_k /4$, the minimality of $N_{k+1}$ gives
\begin{equation}\label{eq22}
3\eta d_{k}/4\ge \sum^{N_{k+1}}_{n=M_{k+1}} |a_{n}|\ge \eta d_{k}/2.
\end{equation}
Let $\xi_{k}\in I_{k}$ such that $|w-F_{k}(\xi_{k})|=d_{k}$. Apply Corollary~\ref{coro3.1} to the point $z_k^* \in \mathbb{D}$, the value~$w-F_{k}(\xi_{k})$ and the block~$P_{k+1}$ to obtain a point~$\xi_{k+1} \in \eta^{-1} I(z_k^*) \subset I_{k}$ such that
\begin{equation*}
\begin{split}
|w-F_{k}(\xi_{k})-P_{k+1}(\xi_{k+1})|&\le d_{k}-\eta \sum_{n\in\mathcal{P}_{k+1}}|a_{n}|\\*[5pt]
&\le d_{k}\left(1-\frac{\eta^{2}}{2}\right).
\end{split}
\end{equation*}
Hence
$$
|w-F_{k}(\xi_{k+1})-P_{k+1}(\xi_{k+1})|\le d_{k}\left( 1-\frac{\eta^{2}}{2}\right)+\alpha_{k},
$$
where $\alpha_{k}=\max \{|F_{k}(\xi)-F_{k}(\xi')|\, : \xi,\xi'\in I_{k}\}$. We deduce 
\begin{equation}\label{eq23}
|w-F_{k+1}(\xi_{k+1})|\leq d_{k}\left( 1-\frac{\eta^{2}}{2}\right)+\alpha_{k}+\sum_{n\in\mathcal{P}_{k}^{*}}|a_{n}|.
\end{equation}
Since $f$ is expanding and $N_{k+1}>N_{k}$, the identity~\eqref{eq21} gives $m(f^{N_{k+1}}(I_{k}))> 1/2$.~Let $I_{k+1}$ be an arc with $\xi_{k+1}\in I_{k+1}\subset I_{k}$ and $m(f^{N_{k+1}}(I_{k+1}))= 1/2$.~We now check \eqref{eq19}, \eqref{eq20} and \eqref{eq21}.~Note that \eqref{eq21} holds by construction. Define $\delta_k=\alpha_k+\sum_{n\in{\cal P}^*_k}|a_n|$. The estimate \eqref{eq23} gives $d_{k+1} \leq d_k (1- \eta^2 /2) + \delta_k$, which is \eqref{eq19}  with $\gamma = 1- \eta^2 /2$. Corollary \ref{coro2.3} gives that $\alpha_k$ tends to zero and since ${\cal P}^*_k$ contains a fix number of indices, it follows that $\delta_k$ tends to zero. Finally note that \eqref{eq22} and the fact that $d_k$ tends to $0$ as $k$ tends to infinity, give \eqref{eq20}.

\vspace*{7pt}
\noindent
(II) Assume $\max\limits_{j\ge N_{k}}|a_{j}|>\eta d_{k}/4$. In this case, we choose $P_{k+1}$ having a single term, that is, $N_{k+1}=M_{k+1}$.~Let $\xi_{k}\in I_{k}$ such that $|w-F_{k}(\xi_{k})|=d_{k}$. As in case (I), let $I_{k+1}$ be an arc with $\xi_{k}\in I_{k+1}\subseteq I_{k}$ and $m(f^{N_{k+1}}(I_{k+1}))= 1/2$. Hence \eqref{eq21} follows by construction. Since in this case $\mathcal{P}_{k}^{*}\cup \mathcal{P}_{k+1}$ has $T+1$~elements, \eqref{eq20} follows from the assumption that $\{a_{n}\}$ tends to~$0$. Define $\delta_{k+1} =\left(4{\eta}^{-1}+T+1\right) \max_{j\ge N_{k}} |a_j|$. Note that $\delta_k$ tends to $0$ as $k$ tends to infinity. Then 
\begin{equation*}
\begin{split}
d_{k+1}\leq |F_{k+1}(\xi_{k})-w|&\le |F_{k}(\xi_{k})-w|+|P_{k+1}(\xi_{k})|+|P_{k}^{*}(\xi_{k})|\\*[5pt]
&\le d_{k}+(T+1)\max_{j\ge N_{k}}|a_{j}| \leq \delta_{k+1}.
\end{split}
\end{equation*}
This gives \eqref{eq19}.
\end{proof}

The proof of Theorem~\ref{theo1.2} has some similarities with the proof of Theorem~\ref{theo1.1} but requires some new ideas.

\begin{proof}[Proof of Theorem~\ref{theo1.2}]
Let $\eta=\eta(f)>0$ be the constant appearing in Corollary~\ref{coro3.1}. Fix any $\xi_{1}\in \partial\mathbb{D}$. Let $N_{1}$ be a positive integer to be fixed later and consider $L_{1}=\sum\limits^{N_{1}}_{n=1}a_{n}f^{n}$. We will show that for any $w\in\mathbb{C}$ satisfying
\begin{equation}\label{eq24}
|w-L_{1}(\xi_{1})|\le \frac{\eta}{10}\sum^{\infty}_{n=N_{1}+1}|a_{n}|,
\end{equation}
there exists $\xi\in\partial\mathbb{D}$ such that $\sum\limits^{\infty}_{n=1}a_{n}f^{n} (\xi)$ converges and its sum is $w$.

Take $M_1 = 1$. We first explain the choice of the positive integer $N_1$. Let $0< \gamma_1 = \gamma_1 (\eta) < 1$ be a  number to be fixed later. Since $f^n$ tends to $0$ uniformly on compacts of $\mathbb{D}$, there exists an integer $T=T(\eta) >0$ such that 
\begin{equation}\label{ite}
|f^T (z)| \leq \eta \quad \text{  if  } \quad  |z|< \gamma_1.
\end{equation}
Also note that the assumption \eqref{eq3} gives
\begin{equation*}
\lim_{j \to \infty} \frac{\sum_{n=j}^{j+T} |a_n|}{\sum_{n=j+T}^{\infty} |a_n|} =0. 
\end{equation*}
Hence there exists an integer $N_1 >0$ such that 
\begin{equation}\label{des1}
 \sum_{n=j}^{j+T} |a_n| \leq \frac{\eta^2}{100} \sum_{n=j+T}^{\infty} |a_n| , \quad j \geq N_1
 \end{equation}
and
\begin{equation}\label{des2}
  |a_j| \leq \frac{\eta^2}{300} \sum_{k>j } |a_k| , \quad  j \geq N_1.
 \end{equation}
Fix any $\xi_1 \in \partial\mathbb{D}$ and let $I_1$ be an arc containing $\xi_1$ such that $m(f^{N_1} (I_1)) = 1/2$. Fix $w \in \mathbb{C}$ satisfying \eqref{eq24}. Note that
\begin{equation*}
d_1 :=\inf \{|w-L_{1}(\xi)|:\xi\in I_{1}\} <\frac{\eta}{10}\sum^{\infty}_{n=N_{1}+1}|a_{n}|. 
\end{equation*}

By induction we will construct a sequence of nested arcs~$I_{k}\subset\partial\mathbb{D}$ and positive integers $M_{k}<N_{k}$ with $M_{1}=1$ and $N_{k}<M_{k+1}$, $k\ge 1$, such that the partial sums $F_1 = L_1$, $\displaystyle F_k=\sum_{j=1}^k L_j+\sum_{j=1}^{k-1} S_j$, $k>1$, where
$$
L_{j}=\sum^{N_{j}}_{n=M_{j}}a_{n}f^{n} ,\quad S_{j}=\sum_{n=N_{j}+1}^{M_{j+1}-1} a_{n}f^{n},\quad j\ge 1,
$$
satisfy
\begin{align}
&m(f^{N_{k}}(I_{k}))= 1/2, \quad k\ge 1,\label{eq25}\\
&d_{k}:=\inf \left\{ \biggl|w-F_k(\xi)\biggr|:\xi\in I_{k}\right\}\le \frac{\eta}{10} \sum^{\infty}_{n=N_{k}+1}|a_{n}|,\quad k\geq 1,\label{eq26}
\end{align}
Note that this would end the proof. Since $m(f^{N_{k}}(I_{k}))= 1/2$, Corollary~\ref{coro2.3} gives
\begin{equation}\label{eq29}
\max \{|F_{k}(\xi)-F_{k}(\xi')|: \xi,\xi'\in I_{k}\}\underset{k\to\infty}{\longrightarrow} 0.
\end{equation}
Pick $\xi \in\bigcap\limits_{ k}I_{k}$. Note that \eqref{eq26} gives that $d_k$ tends to zero. Then \eqref{eq29} gives that $F_k(\xi)$ converges to~$w$. Since $\sum\limits_{n}|a_{n}|<\infty$, this would finish the proof.

The construction of the blocks~$L_{j}$, $S_{j}$ and the arcs~$I_{j}$ is done by induction and has some similarities to the construction used in the proof of Theorem~\ref{theo1.1}. Again the number of terms of each~$S_{j}$ will be a large number independent of~$j$, actually $M_{j+1}-N_{j}=T$ for any $j$, while the number of terms of the blocks~$L_{j}$ may be arbitrarily large, that is, $N_{j}-M_{j}$ may tend to~$\infty$ as $j\to\infty$. However the proof is more involved and several technical adjustments are needed. We have already constructed $F_1 = L_1$ and the arc $I_1$ satisfying \eqref{eq25} and \eqref{eq26} for $k=1$. Assume, by induction, that the blocks~$L_{1},\dotsc, L_{k}$, $S_{1},\dotsc, S_{k-1}$ and the arcs $I_{k}\subset I_{k-1}\subset\dotsb \subset I_{1}$ have been constructed verifying the induction assumptions. We will complete the inductive step constructing the blocks~$S_{k}$, $L_{k+1}$ and the arc~$I_{k+1}$.

Fix a constant $0< c_{1}=c_{1}(\eta)<1$ such that there exists $z_{k}^{*} \in \mathbb{D}$ with $\rho (z_{k}^{*},z(I_{k}))\le c_{1}$ such that $\eta^{-1}I(z_{k}^{*})\subset I_{k}$.~Observe that \eqref{eq25} and part~(b) of Lemma~\ref{lem2.4} provide a constant $0<\gamma_{0}=\gamma_{0}(f)<1$ such that $|f^{N_{k}}(z(I_{k}))|\le \gamma_{0}$.~By Schwarz's Lemma there exists a constant $0<\gamma_{1}=\gamma_{1}(\gamma_{0},\eta)<1$ such that $|f^{N_{k}}(z_{k}^{*})|\le\gamma_{1}$. Since $M_{k+1} = N_k + T$, estimate \eqref{ite} gives  $|f^{M_{k+1}}(z_{k}^{*})|\le \eta$. This defines $S_{k}$ and allows to apply Corollary~\ref{coro3.1} to any partial sum starting at the index~$M_{k+1}$. The construction of the block~$L_{k+1}$ and the arc~$I_{k+1}$ requires to consider two cases:

\vspace*{7pt}
\noindent
(I) Assume
\begin{equation}\label{eq30}
\sum^{\infty}_{n=N_{k}+1}|a_{n}| \le \frac{100 d_{k}}{\eta}.
\end{equation}
Note that the induction hypothesis~\eqref{eq26} gives
$$
\sum^{\infty}_{n=N_{k}+1}|a_{n}|\ge \frac{10d_{k}}{\eta}.
$$
Since $M_{k+1} = N_k + T$, estimate \eqref{des1} gives 
$$
\sum^{\infty}_{n=M_{k+1}}|a_{n}|\ge \frac{9d_{k}}{\eta}.
$$
Let $N_{k+1}$ be the smallest positive integer bigger than $M_{k+1}$ such that
$$
\sum^{N_{k+1}}_{n=M_{k+1}}|a_{n}|\ge \frac{d_{k}\eta}{2}.
$$
This defines the block~$L_{k+1}$. Note that
$$
\sum^{N_{k+1}}_{n=M_{k+1}}|a_{n}|\le \frac{d_{k}\eta}{2}+|a_{N_{k+1}}|.
$$
By estimate \eqref{des2},
$$
 |a_{N_{k+1}}| \leq \frac{\eta^2}{300} \sum\limits_{n>N_{k+1}}|a_{n}|.
$$
Hence estimate (\ref{eq30}) gives
\begin{equation}\label{eq31}
\frac{5d_{k} \eta}{6} \ge \sum^{N_{k+1}}_{n=M_{k+1}}|a_{n}| \ge \frac{d_{k} \eta}{2} .
\end{equation}
Let $\xi_{k}\in  I_{k}$ such that $d_{k}=|w-F_{k}(\xi_{k})|$. Apply Corollary~\ref{coro3.1} to the point~$z_{k}^{*}$, the value $w - F_k (\xi_k)$ and the block~$L_{k+1}$, to find $\xi_{k+1}\in \eta^{-1}I(z_{k}^{*})  \subset I_k $ such that
$$
|w-F_{k}(\xi_{k})-L_{k+1}(\xi_{k+1})|\le d_{k}-\eta \sum^{N_{k+1}}_{n=M_{k+1}}|a_{n}|.
$$
Hence
\begin{equation}\label{eq32}
|w-F_{k}(\xi_{k})-S_{k}(\xi_{k+1})-L_{k+1}(\xi_{k+1})|\le d_{k}-\eta \sum^{N_{k+1}}_{n=M_{k+1}}|a_{n}|+\sum^{M_{k+1}-1}_{n=N_{k}+1}|a_{n}|.
\end{equation}
Since $m(f^{N_{k}}(I_{k}))= 1/2$ and $f$ is expanding, one can find an arc~$I_{k+1}$ with $\xi_{k+1}\in I_{k+1}\subset I_{k}$ and $m(f^{N_{k+1}}(I_{k+1}))= 1/2$. We complete the inductive step showing
\begin{equation}\label{eq33}
d_{k+1}\le \frac{\eta}{10}\sum^{\infty}_{n=N_{k+1}+1}|a_{n}|.
\end{equation}
Let $K=\min \{|f'(\xi)|: \xi \in  \partial\mathbb{D}\}$. Since $f(0)=0$ and $f$ is not a rotation, we have $K>1$. Note that by Lemma~\ref{lem2.1} we have
$$
|F_{k}(\xi_{k})-F_{k}(\xi_{k+1})|\le \pi  \sum^{N_{k}}_{n=1}|a_{n}| K^{n-N_k}.
$$
Hence \eqref{eq32} gives
$$
d_{k+1}\le d_{k}-\eta\sum^{N_{k+1}}_{n=M_{k+1}}|a_{n}| +\sum^{M_{k+1}-1}_{n=N_{k}+1}|a_{n}| + \pi \sum^{N_{k}}_{n=1}|a_{n}|K^{n-N_{k}}.
$$
Applying the inductive assumption, the estimate \eqref{eq33} would follow from   
$$
\frac{\eta}{10}\sum^{\infty}_{n=N_{k}+1}|a_{n}| -\eta\sum^{N_{k+1}}_{n=M_{k+1}}|a_{n}| +\sum^{M_{k+1}-1}_{n=N_{k}+1}|a_{n}|+ \pi \sum^{N_{k}}_{n=1}|a_{n}| K^{n-N_{k}}\le\frac{\eta}{10}\sum^{\infty}_{n=N_{k+1}+1}|a_{n}|.
$$
Hence, we need to check 
\begin{equation}\label{eq34}
\frac{-9\eta}{10}\sum^{N_{k+1}}_{n=M_{k+1}}|a_{n}| +\left(1+\frac{\eta}{10}\right) \sum^{M_{k+1}-1}_{n=N_{k}+1}|a_{n}|+ \pi \sum^{N_{k}}_{n=1}|a_{n}|K^{n-N_{k}}\le 0.
\end{equation}
Now, (\ref{eq31}) says that the first sum is comparable to $d_k$. By Lemma~\ref{lem2.7} and the assumption (\ref{eq30}), last sum of (\ref{eq34}) can be absorved by the fist one.
On the other hand $M_{k+1}=N_{k}+T$ and  hence by \eqref{des1}, the second sum of \eqref{eq34} is also absorved by the first one. We see that \eqref{eq34} holds and the inductive step is completed in this case.

\vspace*{7pt}
\noindent
(II) Assume
\begin{equation}\label{eq35}
\sum^{\infty}_{n=N_{k}+1}|a_{n}| \ge \frac{100d_{k}}{\eta}.
\end{equation}
Let $\xi_k\in I_k$ such that $d_k=|w-F_k(\xi_k)|$ and pick $N_{k+1}=M_{k+1}$, that is, in this case $L_{k+1}$ has a single term. Since $m(f^{N_{k}}(I_{k}))= 1/2$ and $f$ is expanding, one can take an arc~$I_{k+1}$ with $\xi_{k}\in I_{k+1}\subset I_{k}$ and $m(f^{N_{k+1}}(I_{k+1}))= 1/2$. Note that
$$
d_{k+1}\le d_{k}+\sum^{M_{k+1}-1}_{n=N_{k}+1}|a_{n}|+|a_{N_{k+1}}|.
$$
Applying \eqref{eq35} we deduce
$$
d_{k+1}\le \frac{\eta}{100}\sum^{\infty}_{n=N_{k}+1}|a_{n}|+\sum^{M_{k+1}}_{n=N_{k}+1}|a_{n}|
$$
and to complete the inductive step, it is sufficient to check
$$
\frac{\eta}{100}\sum^{\infty}_{n=N_{k}+1}|a_{n}|+\sum^{M_{k+1}}_{n=N_{k}+1}|a_{n}|\le \frac{\eta}{10}\sum^{\infty}_{n=N_{k+1}+1}|a_{n}|.
$$
The choice $M_{k+1}-N_{k}=T+1$ and estimate \eqref{des1} finish the proof. 
\end{proof}
\section{Concluding remarks}\label{sec4}

Given $\xi \in \partial \mathbb{D}$ and $M>1$, let $\Gamma (\xi , M) = \{z \in \mathbb{D} : |z - \xi| \leq M(1-|z|) \}$ be the Stolz angle with vertex at $\xi$ and aperture depending on $M$. 
A function $F:\mathbb{D} \rightarrow \mathbb{C}$ has a non tangential limit at the point $\xi \in \partial \mathbb{D}$, which will be denoted by $\lim_{{z\to \xi}\atop{\nprec}} F(z)$, if for any $M>1$ the limit 
$$\lim_{ z \in \Gamma (\xi , M), z \to \xi } F(z)$$ 
exists and is finite. 

In the statements of Theorems~\ref{theo1.1} and \ref{theo1.2} one can replace the sum $\sum\limits_{n}a_{n}f^{n}(\xi)$ by the non-tangential limit
$$
\lim_{{z\to \xi}\atop{\nprec}}\sum_{n}a_{n}f^{n}(z).
$$
This follows from our next result which can be understood as a version of the classical Abel's Theorem in our context.

\begin{theorem}\label{theo4.1}
Let $f$ be a finite Blaschke product with $f(0)=0$ which is not a rotation. Let $\{a_{n}\}$ be a sequence of complex numbers and let $\xi\in\partial\mathbb{D}$ such that $\sum_n  a_{n}f^{n}(\xi)$ converges. Then the non-tangential limit 
$$
\lim_{{z\to\xi}\atop{\nprec}} \displaystyle{\sum_{n=1}^{\infty}} a_{n}f^{n}(z)
$$
exists and it is equal to $\sum\limits_{n=1}^\infty a_{n}f^{n}(\xi)$.
\end{theorem}

\begin{proof}
Let $K_1=\max\{|f'(\xi)|\, :\, \xi\in\partial\mathbb{D}\}$ and pick $M> K_1$. Fix $0< \delta < (10 K_1)^{-1}$. For $z\in \Gamma(\xi, M)$ let $N(z)$ be the smallest positive integer such that $m(f^{N(z)}(I(z)))\ge \delta$. Apply Lemma~\ref{lem4.0} to obtain a constant $C=C(\delta , f) >0$ such that
\begin{equation}\label{eq38}
\sum^{N(z)}_{k=1}|f^{k}(\xi)-f^{k}(z)|\le C, \quad \xi \in I(z). 
\end{equation}
Note that $N(z)\to\infty$ as $|z|\to 1$ and $m(f^{N(z)}(I(z)))\le \delta K_{1}$.~Part~(b) of Lemma~\ref{lem2.4} gives a constant $0<\gamma=\gamma(\delta,f)<1$ such that $|f^{N(z)}(z)|\le \gamma$. Pommerenke estimates of the Denjoy--Wolff Theorem (Lemma~2 of \cite{P}) provide constants $C_1=C_1(\gamma) >0$ and $0<a=a(f)<1$ such that
\begin{equation}\label{eq37}
|f^{n}(z)|\le C_1 a^{n-N(z)},\quad n\ge N(z).
\end{equation}

We are now ready to prove the result. Fix $\varepsilon>0$.~We want to show that there exists $N_{0}=N_{0}(\varepsilon)>0$ such that for any $z\in\Gamma(\xi , M)$ sufficiently close to~$\xi$, we have
\begin{equation}\label{eq39}
\left|\sum^{\infty}_{n=N}a_{n}(f^{n}(z)-f^{n}(\xi))\right|\le\varepsilon,\quad N\ge N_{0}.
\end{equation}
Summing by parts
$$
\sum^{\infty}_{n=N}a_{n}(f^{n}(z)-f^{n}(\xi))=\sum^{\infty}_{n=N+1}S_{n}(\xi) \left(\frac{f^{n}(z)}{f^{n}(\xi)}-\frac{f^{n-1}(z)}{f^{n-1}(\xi)}\right)+S_{N}(\xi)\left(\frac{f^{N}(z)}{f^{N}(\xi)}-1\right),
$$
where
$$
S_{n}(\xi)=\sum^{\infty}_{k=n}a_{k}f^{k}(\xi).
$$
Next, we will show that there exists a constant $C_2=C_2 (\delta, f)>0$ such that
\begin{equation}\label{eqL}
    \sum_{n=1}^\infty \left|\frac{f^n(z)}{f^n(\xi)}-\frac{f^{n-1}(z)}{f^{n-1}(\xi)}\right|\leq C_2.
\end{equation}
Note that estimate \eqref{eq37} gives a constant $C_3=C_3(\delta, f)>0$ such that,
$$\sum_{n=1}^\infty \left|\frac{f^n(z)}{f^{n}(\xi)}-\frac{f^{n-1}(z)}{f^{n-1}(\xi)}\right|\leq \sum_{n=1}^{N(z)} \left|f^n(z) f^{n-1}(\xi)- f^{n-1}(z)f^{n-1}(\xi) \right|+C_3.$$
Adding and substracting ${f^n(\xi)}{f^{n-1}(\xi)}$ and applying estimate \eqref{eq38}, we obtain \eqref{eqL}. By assumption, we can choose $N_0$ such that  $|S_{n}(\xi)|\le\varepsilon$ if $n\ge N_{0}$. This gives \eqref{eq39} and finishes the proof.
\end{proof}

Another turn of the screw of the methods in the proof of Theorem~\ref{theo1.1} leads to our next result which is formally, a generalization of Theorem~\ref{theo1.1}. An analogous result in the context of lacunary series can be found in \cite{KWW}. Given a sequence of complex numbers $\{b_n \}$ let $\{\sum b_n\}'$ denote the set of accumulation points of its partial sums, that is, 
$$
\{\sum b_n\}' = \bigcap_{n=1}^\infty \overline{\left\{\sum_{k=1}^m b_k\, :\, m\geq n\right\}}.
$$
Note that if $\{b_n \}$ tends to $0$, then $\{\sum b_n\}' $ is a closed connected set of the extended complex plane. 
\begin{theorem}\label{theoremA}
Let $f$ be a finite Blaschke product with $f(0)=0$ which is not a rotation. Let $\{a_{n}\}$ be a sequence of complex numbers tending to~$0$ such that $\sum |a_{n}|=\infty$. Then, for any closed connected set $K$ of the extended complex plane, there exists a point $\xi\in\partial\mathbb{D}$ such that
$$
K = \{\sum a_n f^n (\xi) \}' .
$$
\end{theorem}

\begin{proof}[Proof of Theorem~\ref{theoremA}]
Since the proof only requires an iteration of the ideas of the proof of Theorem~\ref{theo1.1}, we will only sketch it. Given the set $K$, let us consider a sequence $\{w_n\}$ of complex numbers with $|w_{n+1} - w_n| \to 0$ as $n \to \infty$, such that
$$K=\bigcap_{n=1}^\infty\overline{\{w_m\, :\, m\geq n\}}.$$ 
We use the notation  
$
S_N (\xi) = \sum_{n=1}^N a_n f^n (\xi)$. Using the methods in the proof of Theorem \ref{theo1.1} one can construct a closed arc $I_1\subset\partial\mathbb{D}$ and an integer $N_1$ so that for any $\xi\in I_1$, the partial sum $S_{N_1} (\xi) $ is close to $w_1$ and $S_m (\xi)$ is not far from the segment $[0,w_1]$ for any $0 < m\leq N_1$. Once one has constructed this arc, using again the same argument, one constructs a closed subarc $I_2\subset I_1$ and an integer $N_2>N_1$ such that for any $\xi\in I_2$,  the partial sum $S_{N_2} (\xi)$ is close to $w_2$ and $S_{m} (\xi)$ is not far from the segment $[w_1,w_2]$ for any $N_1\leq m\leq N_2$. Iterating this process and taking $\xi \in \bigcap I_j$ one obtains the desired result.
\end{proof}


Theorem~\ref{theoremA} has the corresponding counterpart in the context  of radial behaviour of the function $\displaystyle \sum a_nf^n$. Given an analytic function $g:\mathbb{D}\to\mathbb{C}$, the radial cluster set $\text{Cl}_r(g,\xi)$ of the function $g$ at the point $\xi\in\partial\mathbb{D}$ is defined to be
$$
\text{Cl}_r(g,\xi)=\bigcap_{r<1} \overline{\left\{g(s\xi)\, :\, s\geq r\right\}}.
$$
We will need the following easy consequence of Lemma~\ref{lem4.0}. 

\begin{lemma}\label{csr}
Let $f$ be a finite Blaschke product with $f(0)=0$ which is not a rotation. Fix $0< \delta < ( 10 \max \{|f'(\xi)| : \xi \in \partial \mathbb{D}\})^{-1}$. Then there exits a constant $C=C(\delta , f) >0$ such that the following statements hold. 

(a) Given $z \in \mathbb{D}$ let $N(z)$ be the smallest positive integer such that $m(f^{N(z)}(I(z)))\ge \delta$. Then 
\begin{equation*}
    \left| \sum_{n=1}^\infty a_n f^n (z) -  \sum_{n=1}^{N(z)} a_n f^n (\xi) \right| \leq C \sup_n |a_n|,\quad \xi \in I(z).
\end{equation*}

(b) Fix $\xi \in \partial \mathbb{D}$. Given a positive integer $N$, consider $r_N = \sup \{0<r<1\ : m(f^N (I(r \xi))) \geq \delta \}$. Then 

\begin{equation*}
    \left| \sum_{n=1}^N a_n f^n (\xi) -  \sum_{n=1}^{\infty} a_n f^n (r_N \xi) \right| \leq C \sup_n |a_n| .
\end{equation*}
\end{lemma}

\begin{proof}
Apply Lemma~\ref{lem4.0} to obtain a constant $C=C(\delta , f) >0$ such that
\begin{equation*}
\sum^{N(z)}_{n=1} |a_n| |f^{n}(\xi)-f^{n} (z)|\le C \sup_n |a_n| , \quad \xi \in I(z). 
\end{equation*}
The exponential decay \eqref{eq37} given by Pommerenke's result, provides a constant $C_1=C_1 (\delta) >0$ such that 
\begin{equation*}
   \sum_{n= N(z)}^{\infty} | f^n (z) | \leq C_1 .
\end{equation*}
This finishes the proof of (a). Part (b) follows similarly. 
\end{proof}

Let $\{a_n\}$ be a sequence of complex numbers tending to $0$ and $F = \sum a_n f^n$. Note that in the previous Lemma we have that $N(z) \to \infty$ as $|z| \to 1$ and $r_N \to 1$ as $N \to \infty$. We deduce that $\text{Cl}_r(F,\xi) = \{\sum a_n f^n (\xi)\}'$ for any $\xi \in \partial \mathbb{D}$. Hence next result follows from Theorem~\ref{theoremA}. 

\begin{theorem}\label{theoremB}
Let $f$ and $\{a_n\}$ be as in  Theorem \ref{theoremA} and let $F$ be the analytic function defined by $\displaystyle F(z)=\sum_{n=1}^\infty a_n f^n(z)$, $ z \in  \mathbb{D} $. Then, for any closed connected set $K$ of the extended complex plane there exists a point $\xi\in\partial\mathbb{D}$ such that
$Cl_r(F,\xi)=K.$
\end{theorem}
\vspace{0.3cm}
\noindent We finally mention some related open problems we have not explored.
\begin{enumerate}
\item Let the function $f$ and the coefficients~$\{a_{n}\}$ be as in the statement of Theorem~\ref{theo1.1}. Fix $w\in\mathbb{C}$.~It is natural to study the size of the set~$E(w)$ of points~$\xi\in\partial\mathbb{D}$ such that $\sum a_{n}f^{n}(\xi)=w$. Note that if we assume $\sum |a_{n}|^{2}=\infty$, then $\sum a_{n}f^{n}$ is a function in the Bloch space that has radial limit at almost no point of the unit circle (see~\cite{N}). In this situation Rohde proved that the set 
$$
\left\{\xi\in\partial\mathbb{D}:\lim_{r\to 1}\sum_{n=1}^{\infty} a_{n}f^{n}(r\xi)=w\right\}
$$
has Hausdorff dimension~$1$. See \cite{Ro}. It is also natural to ask for the size of the set of points $\xi \in \partial \mathbb{D}$ where the conclusion of Theorem~\ref{theoremA} holds.

\item Let $F(z)=\sum a_{n}z^{k_{n}}$ be a lacunary power series with radius of convergence equals to~$1$, such that $\sum |a_{n}|=\infty$. Murai proved that for any $w\in\mathbb{C}$ there are infinitely many $z\in\mathbb{D}$ such that $F(z)= w$ (\cite{Mu3}). It is natural to seek for analogus results in our context. More concretely, let $f$ and $\{a_{n}\}$ be as in the statement of Theorem~\ref{theo1.1}.~Is it true that for any $w\in\mathbb{C}$ there exists $z\in\mathbb{D}$ such that $\sum a_{n}f^{n}(z)=w$?

\item The High Indices Theorem of Hardy and Littlewood provides a converse of Abel's Theorem in the setting of lacunary series. See \cite{K}. Hence it is natural to ask for the converse of Theorem~\ref{theo4.1}.
\end{enumerate}


\begin{thebibliography}{KWW}


\bibitem[Ba]{Ba}
\textsc{Bara\'nski, K.}
On some lacunary power series.
\emph{Michigan Math.\ J.} \textbf{54} (2006), no.~1, 65--79.

\bibitem[Be]{Be}
\textsc{Belov, A.\ S.}
On the Salem and Zygmund problem with respect to the smoothness of an analytic function that generates a Peano curve. (Russian)
\emph{Mat.\ Sb.} \textbf{181} (1990), no.~8, 1048--1060; translation in \emph{Math.\ USSR-Sb.} \textbf{70} (1991), no.~2, 485--497.

\bibitem[KWW]{KWW}
\textsc{Kahane, J.-P.; Weiss, M.; Weiss, G.}
On lacunary power series.
\emph{Ark.\ Mat.} \textbf{5} (1963), 1--26.

\bibitem[K]{K}
\textsc{Korevaar, J.}
\emph{Tauberian theory. A century of developments}.
Grundlehren der Mathematischen Wissenschaften [Fundamental Principles of Mathematical Sciences]~\textbf{329}.
Springer-Verlag, Berlin, 2004.

\bibitem[Mu1]{Mu1}
\textsc{Murai, T.}
The value-distribution of lacunary series and a conjecture of Paley.
\emph{Ann.\ Inst.\ Fourier (Grenoble)} \textbf{31} (1981), no.~1, vii, 135--156.

\bibitem[Mu2]{Mu2}
\textsc{Murai, T.}
On lacunary series.
\emph{Nagoya Math.\ J.} \textbf{85} (1982), 87--154.

\bibitem[Mu3]{Mu3}
\textsc{Murai, T.}
The boundary behaviour of Hadamard lacunary series.
\emph{Nagoya Math.\ J.} \textbf{89} (1983), 65--76.

\bibitem[N]{N}
\textsc{Nicolau, A.}
Convergence of linear combinations of iterates of an inner function. Preprint, 2021.

\bibitem[NS]{NS}
\textsc{Nicolau, A.; Soler i Gibert, O.}
A central limit theorem for inner functions. Preprint, 2020.

\bibitem[P]{P}
\textsc{Pommerenke, Ch.}
On ergodic properties of inner functions.
\emph{Math.\ Ann.} \textbf{256} (1981), no.~1, 43--50.

\bibitem[Ro]{Ro}
\textsc{Rohde, S.}
The boundary behavior of Bloch functions.
\emph{J. London Math. Soc.} \textbf{48} (1993), 488--499.

\bibitem[SZ]{SZ}
\textsc{Salem, R.; Zygmund, A.}
Lacunary power series and Peano curves.
\emph{Duke Math.\ J.} \textbf{12} (1945), 569--578.

\bibitem[W]{W}
\textsc{Weiss, M.}
Concerning a theorem of Paley on lacunary power series.
\emph{Acta Math.} \textbf{102} (1959), 225--238.

\bibitem[Y]{Y}
\textsc{Younsi, M.}
Peano curves in complex analysis.
\emph{Amer.\ Math.\ Monthly} \textbf{126} (2019), no.~7, 635--640.


\bibitem[Z]{Z}
\textsc{Zygmund, A.}
\emph{Trigonometric series}. Vol.~I, II. Second edition, reprinted with corrections and some additions.
Cambridge University Press, London-New York, 1968.

\end{thebibliography}
\end{document}